\documentclass[12pt]{amsart}
\usepackage{graphicx} 
\usepackage{amsmath}
\usepackage{amssymb}
\usepackage{anysize}
\usepackage{xcolor}
\usepackage{hyperref}
\usepackage{enumitem}
\theoremstyle{plain}
\newtheorem{theor}{Theorem}[section]
\newtheorem{Theor}{Theorem}[section]{}
\newtheorem{lemma}{Lemma}[section]

\newtheorem{corol}{Corollary}[section]
\theoremstyle{definition}
\newtheorem{defin}{Definition}[section]

\newtheorem{notat}{Notation}[section]
\newtheorem{quest}{Question}[section]

\newtheorem*{claim}{Claim}

\DeclareMathOperator{\cof}{\bold{cof}}
\DeclareMathOperator{\ran}{\bold{range}}

\DeclareMathOperator*{\suc}{\bold{successor}}

\newcommand{\AND}{\text{ and }}

\newcommand{\angbr}[1]{\langle #1 \rangle}

\newcommand{\forces}[2]{\Vdash_{#1} \mbox{``} #2 \mbox{''}}

\newcommand{\SetOf}[2]{\left\{#1 \ \left| \ #2 \right.\right\}}

\newcommand{\Poset}{{\mathbb P}}

\author[Raghavan]{Dilip Raghavan}
\thanks{Research leading to the results of this paper was carried out when the first author was a Fields Visiting Scholar.
The first author thanks the Fields Institute for its kind hospitality.
The first author was partially supported by the Singapore Ministry of Education's research grant number A-8001467-00-00}
	\address[Raghavan]{Department of Mathematics\\
		National University of Singapore\\
		Singapore 119076.} 
\email{\href{dilip.raghavan@protonmail.com}{dilip.raghavan@protonmail.com}}
\urladdr{\url{https://dilip-raghavan.github.io/}}
\author[Stepr{\= a}ns]{Juris Stepr{\= a}ns}
\thanks{The second author is partially supported by an NSERC Discovery grant}
\address[Stepr{\= a}ns]{Department of Mathematics and Statistics\\ York University\\ 4700 Keele St.\@\\ Toronto, ON M3J 1P3, Canada.}
\email{\href{steprans@yorku.ca}{steprans@yorku.ca}}
\urladdr{\url{https://www.yorku.ca/science/profiles/faculty/juris-steprans/}}

\title{Adding ultrafilters to Shelah's model for no nowhere dense ultrafilters}
\date{June 2024}

\begin{document}
\begin{abstract}
 We exhibit a forcing for producing a model with no nowhere dense ultrafilters that satisfies the full Sacks Property.
 By interleaving this forcing with other forcing notions, a model containing a $(2, {\aleph}_{0})$-selective ultrafilter, but no nowhere dense ultrafilters is produced.
 It is thus proved that the existence of $(2, {\aleph}_{0})$-selective ultrafilters does not imply the existence of nowhere dense ultrafilters.
\end{abstract}
\maketitle
\section{Introduction}\label{ievads}
The history of obtaining models of set theory with no P-points consists of an interesting sequence of incremental results, most of which are still of interest today in spite of stronger subsequent results.
The sequence begins with Kunen's proof \cite{MR0427070} that there are no selective ultrafilters after adding $\aleph_2$ random reals to a model  of the Continuum Hypothesis.
While this result was later improved by the celebrated models of Shelah that do not even contain P-points \cite{MR1623206}, the question of whether there are actually P-points in the random model remains of interest (see Dow~\cite{dow}). 
Following Shelah's results, it was shown by Mekler in \cite{MR759670}  that in the original model of Shelah there are not even any P-measures.
In a recent work, Borodulin-Nadzieja, Cancino-Manr{\'i}quez and Morawski~\cite{arXiv:2401.14042} have produced a model where there are no P-points, but P-measures exist.

From these considerations researchers were led to ask the question of whether ultrafilters weaker than P-points, but with similar properties, can be shown to exist without assuming extra set theoretic axioms.
A well known major accomplishment in the positive direction is Kunen's $\mathrm{ZFC}$ construction of weak P-points \cite{MR588822}.
The Tukey ordering provides another path to a weaker notion.
P-points are not Tukey maximal among ultrafilters on $\omega$, and it is not known whether a Tukey non-maximal ultrafilter can be constructed in $\mathrm{ZFC}$.
The reader may consult~\cite{sanu-survey} for recent results about the Tukey ordering on ultrafilters.

However weakening the P-point notion along yet different lines is known not to produce definitions of ultrafilters whose existence can be established without appealing to extra axioms.
The following definition was motivated by van Douwen in \cite{MR0627526} and studied by Baumgartner in \cite{MR1335140} and provides a large class of weakenings of the P-point property.
\begin{defin}\label{BGDJun2a}
If $\mathcal I$ is an ideal on the set $X$ then an ultrafilter
$\mathcal{U}$ is known as an \emph{$\mathcal I$-ultrafilter} if for every
function $F:\omega\to X$ there is some $A\in \mathcal{I}$ such that
$F^{-1}(A)\in \mathcal U$.
\end{defin}
Observe that the definition is equivalent to saying that for every $F: \omega \rightarrow X$, there exists $B \in \mathcal{U}$ so that $F\left[B\right] \in \mathcal{I}$.
Letting $\mathbb F$ be the ideal of finite subsets 
it is not hard to see that an ultrafilter is a P-point if and only if it is an $\mathbb F\otimes\mathbb F$-ultrafilter.
Note also that the larger the ideal $\mathcal I$ is, the weaker the property of being an $\mathcal I$-ultrafilter. So, for example, if $\mathcal I = \mathcal U^*$ then an ultrafilter $\mathcal V$ is a $\mathcal I$-ultrafilter if and only if $\mathcal U\not\leq_{\sf RK}\mathcal V$ and, hence, $\mathcal I$-ultrafilters abound in this case.
But for large, but not maximal, definable ideals such as $\mathbb{NWD}$,  the ideal of nowhere dense subsets of $\mathbb Q$,  the question of the existence of $\mathbb{NWD}$-ultrafilters --- also known as nowhere dense ultrafilters --- becomes more interesting.
After long being open, Shelah was able to prove the consistency of no nowhere dense ultrafilters in  \cite{MR1690694}.

The techniques used in \cite{MR1690694} were quite different from those used in the earlier constructions of models without P-points and were more reminiscent of the ideas introduced in \cite{MR1077075} and \cite{MR1175937}. It very soon became apparent that the model of
\cite{MR1690694} eliminated more types of ultrafilters than just the nowhere dense ones. For example, in that model there are also no
ultrafilters with Property~M ---
an ultrafilter $\mathcal{U} $ will be said to have {\em Property~M} if
for all $\epsilon>0$ and all sequences  $\{A_i\}_{i\in\omega}$ of sets of Lebesgue measure greater than $\epsilon$ there is $X\in \mathcal{U}$ such that $\bigcap_{n\in X}A_n\neq \varnothing$.
The failure of ultrafilters to have this weak P-point type property --- it  is shown in \cite{MR1682072} that ultrafilters with Property~M need not be P-points --- might lead one to conjecture that Shelah's model for no nowhere dense ultrafilters will not harbour any ultrafilters with certain weak P-point type properties.
In fact, the model in \cite{MR1690694} does not contain $\mathcal{I}$-ultrafilters for any $\mathcal{I}$ that is Kat{\v e}tov below $\mathbb{NWD}$.
For example, there are no discrete ultrafilters in this model; in other words, there are $\mathcal{I}$-ultrafilters for $\mathcal{I} = \{A \subseteq {2}^{\omega}: A \ \text{is discrete}\}$.
It remains open whether the existence of nowhere dense ultrafilters implies the existence of discrete ultrafilters, although Shelah proved in~\cite{MR1623206} that it does not imply the existence of P-points.
The reader may consult Brendle~\cite{MR1729447} for further information on the spectrum of $\mathcal{I}$-ultrafilters falling between P-points and nowhere dense ultrafilters.

The evidence does not seem to support this conjectured dearth of special ultrafilters in Shelah's model.
This paper has two goals.
The first goal will be addressed in \S\ref{Ex} which will provide an exposition of the model of \cite{MR1690694} with some simplifications and streamlined notation. However, the reader who hopes to use the ideas of \cite{MR1690694} is advised to consult the original since it contains subtleties, not needed for the result about nowhere dense ultrafilters, which may prove to have uses elsewhere.
The simplifications will demonstrate that a forcing with the Sacks Property, which Shelah's original forcing from \cite{MR1690694} does not have, can be used to kill all nowhere dense ultrafilters.
The second goal is to show how to modify the model of \cite{MR1690694} to obtain models containing ultrafilters  weakening the notion of selectivity, which were studied by Barto\v{s}ov\'{a} and Zucker in \cite{MR3946661}.
The main result of \S\ref{VarsrSect} will prove that the existence of ultrafilters satisfying the weak form of selectivity considered by Barto\v{s}ov\'{a} and Zucker does not imply the existence of any nowhere dense ultrafilters.
It is not yet known whether these can be shown to exist without assuming some extra axioms of set theory.

\section{Review of Shelah's model with no nowhere dense ultrafilters}
\label{Ex}
\subsection{Definition of the partial order}
A key part of the argument of \cite{MR1690694}
 establishing the consistency of no nowhere dense ultrafilters hinged on the Weak Sacks Property. This is a
 simpler instance of the more general PP property of \S2.12A of Chapter~VI of \cite{MR1623206} which has 
 various applications and is bound to have many as yet undiscovered applications as well. However, in this exposition it will be shown how to use the much better understood Sacks Property to achieve the same results, even though the Weak Sacks Property will be needed in \S\ref{VarsrSect}. The following partial order is at the heart of the argument.

\begin{defin}\label{Sheas31ah7ib}
	If $\mathcal {I}$ is an ideal on $\omega$ define
	$\Poset(\mathcal {I})$ to consist of all functions $\sigma$
	defined on $\omega$	 such that  there is $d:\omega\to \omega$ satisfying that for each $n\in \omega$ 
	\begin{enumerate}
		\item $d(n) \leq n$
		\item $\sigma(n) : [d(n),n)\to 2$\label{Excav111}
		\item $d^{-1}(n) \in \mathcal {I}$.
	\end{enumerate}
	For $\sigma \in \Poset(\mathcal {I})$ let $d_\sigma $ denote the function 
	witnessing that $\sigma \in \mathbb P(\mathcal {I})$.
 Note that ${d}_{\sigma}$ is uniquely determined by $\sigma$.
	For $\sigma$ and $\tau$ in 	$\Poset(\mathcal {I})$ define
	$\sigma\leq \tau$ if 
	\begin{enumerate}[resume]
		\item $\sigma(n)\supseteq \tau(n)$ for all $n\in \omega$
		\item\label{BDragJG5aaa} there is $e:\omega \to \omega$ such that $e(n)\leq  n$
		and $d_\sigma(n) = e\circ d_\tau(n)$ 
		for all $n$ 
		\item \label{BDragJG5}if $d_\tau(n) = d_\tau(m) $  then 
		$\sigma(n)\restriction [d_\sigma(n), d_\tau(n)) = \sigma(m)\restriction [d_\sigma(n), d_\tau(n))$ or, in other words,
		if $d_\tau(n) = d_\tau(m)=z $  then 
		$\sigma(n)\restriction [e(z), z) = \sigma(m)\restriction [e(z),z)$.
	\end{enumerate} 
	
\end{defin}

\begin{lemma}\label{DRaudzAin1}
	$ \Poset(\mathcal {I})$ is a partial order.
\end{lemma}
\begin{proof}
	It only needs to be verified that $ \Poset(\mathcal {I})$ is transitive, so suppose that
	$\sigma\leq \tau\leq \theta$. Then there are $e$ and $\bar{e}$ such that
	$d_\tau = e\circ d_\theta$ and $d_\sigma = \bar{e}\circ d_\tau$. To see that Condition~(\ref{BDragJG5}) holds suppose that
	$d_\theta(n) = d_\theta(m) =z$ and
	$\bar{e}\circ e(z)\leq i < z$. Then
	if $ e(z)\leq i < z$ it follows that
	$\sigma(m)(i) = \tau(m)(i) = \tau(n)(i) = \sigma(n)(i)$. On the other hand, if 
	$ \bar{e}( e(z))\leq i < e(z)$
	then $d_\tau(n)  = e(z) = 
	d_\tau(m)$ and hence $\sigma(m)(i) = \sigma(n)(i)$ because $\sigma\leq \tau$.
\end{proof}

\begin{defin}\label{Janh68938h}
	
	For $\sigma$ and $\tau$ in 	$\Poset(\mathcal {I})$ 
	and $a\in [\omega]^{<\aleph_0}$ 
	define
	$\sigma\leq_a \tau$ if
	$\sigma \leq \tau$ and 	$d_\sigma^{-1}\{i\} = d_\tau^{-1}\{i\} $
	for all $i\in a$.
\end{defin}

\begin{defin}\label{SecfgUja682j}
	For $\sigma\in \Poset(\mathcal {I})$, $S\in [\omega]^{<\aleph_0}$ and ${g}\in \prod_{j\in S}2^j$
	define 
	$\sigma[g]$ by
	$$\sigma[g](\ell) = \begin{cases}	
	  	\sigma(\ell) & \text{ if } d_\sigma(\ell)\notin S\\
		{g}(d_\sigma(\ell))\cup \sigma(\ell) &  \text{ if } d_\sigma(\ell) \in S .\\
	\end{cases}$$
If $S\subseteq \omega$ and  $\vec{g}$ is defined on
$ d_\sigma^{-1}(S)$ and $\vec{g}(i) \in 2^j$ whenever $d_\sigma(i) = j\in S$ then let
$\sigma[\vec{g}]$ be defined by
$$\sigma[\vec{g}](\ell) = \begin{cases}	
	\sigma(\ell) & \text{ if } d_\sigma(\ell)\notin S\\
	\vec{g}(d_\sigma(\ell))\cup \sigma(\ell) &  \text{ if } d_\sigma(\ell) \in S .\\
\end{cases}$$
\end{defin}

While the following lemma concluding that $\sigma[{g}]\leq \sigma$  is immediate,
 it may not be the case that $\sigma[\vec{g}]\leq \sigma$ because 
Condition~(\ref{BDragJG5}) of Definition~\ref{Sheas31ah7ib} may fail. 
For example, note that it is not claimed 
in Lemma~\ref{2a9Tu8777hDrh75}
that $\sigma[\vec{g}]\leq \sigma$
 only that 
$		{\sigma}[\vec{g}] \leq {\tau}$ since
this is a case where
Definition~\ref{Sheas31ah7ib} may fail. 
When it is necessary to use  $S$ in the context of Definition~\ref{SecfgUja682j} it will usually be a singleton, but there is a crucial point, Corollary~\ref{CCOOR777hDrh75}, at which an infinite $S$ will be needed.
\begin{lemma}\label{Tu8777hDrhjCCI}
	If 	$\sigma\in \Poset(\mathcal I)$ and $n\in\omega$ and
	${g} \in \prod_{j\in n}2^j$ then $\sigma[g]\leq \sigma$. 		
\end{lemma}
\subsection{Properties of the partial order}
This section will establish that the partial orders are proper and have the Sacks Property. This is the key difference between the partial order to be presented here and that of Shelah's original argument; as will be seen in \S\ref{renaj6}, the partial order of \cite{MR1690694} does not enjoy the Sacks Property.
\begin{lemma}
	\label{2a9Tu8777hDrh75}
	Suppose that
 \begin{itemize}
 \item $\tau\in \Poset(\mathcal I)$ 
\item $ N\in \omega $
 \item  $\ran(d_{\tau})\cap N = A$ 
 \item    $D\subseteq \Poset(\mathcal I)$ is a dense set.
 \end{itemize}
 Then there are
	$\sigma\leq_{A} \tau$ and   $W\subseteq D$ 
 and $\vec{g}:  d_\sigma^{-1}\{N\} \to 2^{N}$
 such that:
	\begin{itemize}
	 \item 
	$|W|\leq 2^{|A|\max(A)}$ 
	\item   
$		{\sigma}[\vec{g}] \leq {\tau}$ 
 \item  $W$ is predense below ${\sigma}[\vec{g}]$.
	\end{itemize}
\end{lemma}
\begin{proof}
	Let $a= A\setminus \{0\}$ and let  $\{g_k\}_{k\in M}$ enumerate 
 $\prod_{j\in a}2^{j}$ and observe that
 $M\leq 2^{|A|\max(A)}$.
 Now construct inductively conditions
$\sigma_k$ for $k\leq M$ such that 
\begin{itemize}
\item  $\sigma_0= \tau$
	\item $\sigma_{k+1} \leq_{a} \sigma_k$
	\item  $\sigma_{k+1}\left[g_k \right]
	\in D$.
\end{itemize}
To complete the inductive argument suppose that
$\sigma_k$ has been found. 
Then $\sigma_k[g_k]\leq \sigma_k$ by Lemma~\ref{Tu8777hDrhjCCI} and so it is possible to find  $\bar{\sigma}\leq \sigma_k[g_k]$ such that
$\bar{\sigma}\in D$. Then let
$\sigma_{k+1}$ be defined by
$$\sigma_{k+1}(n) = \begin{cases}
\bar{\sigma}(n) &\text{ if } d_{\bar{\sigma}}(n)\notin a\\
0_n \cup \bar{\sigma}(n) &\text{ if } d_{\bar{\sigma}}(n)\in a \AND d_{\sigma_k}(n)\notin a\\
\sigma_k(n) &\text{ if }  d_{\sigma_k}(n)\in a
\end{cases}$$
where $0_n$ denotes the function with domain  $d_{\bar{\sigma}}(n)$ and constant value 0. 
To see  that $\sigma_{k+1}[g_k]\leq \bar{\sigma}$ simply note that
$$\sigma_{k+1}[g_k](n) = \begin{cases}
	\bar{\sigma}(n) &\text{ if } d_{\bar{\sigma}}(n)\notin a\\
	0_n \cup \bar{\sigma}(n) &\text{ if } d_{\bar{\sigma}}(n)\in a \AND d_{\sigma_k}(n)\notin a\\
	\sigma_k(n)\cup g_k(n)=\bar{\sigma}(n) &\text{ if }  d_{\sigma_k}(n)\in a .
\end{cases}$$
 It follows that $\sigma_{k+1}[g_k]\in D$.

To see  that $\sigma_{k+1}\leq_a  {\sigma}_k$ it suffices to show that $\sigma_{k+1}\leq  {\sigma}_k$.
$$d_{\sigma_{k+1}}(n) = 
\begin{cases}
d_{\bar{\sigma}}(n) &\text{ if } d_{\bar{\sigma}}(n)\notin a\\
0 &\text{ if } d_{\bar{\sigma}}(n)\in a \AND d_{\sigma_k}(n)\notin a\\
d_{\sigma_k}(n) &\text{ if }  d_{\sigma_k}(n)\in a .
\end{cases}$$
In order to verify Condition~(\ref{BDragJG5}) of Definition~\ref{Sheas31ah7ib} suppose that $d_{\sigma_k}(n) = d_{\sigma_k}(m)$ and $d_{\sigma_{k+1}}(n) \leq i <
d_{\sigma_k}(n)$.
It is immediate that $d_{\sigma_k}(n) = d_{\sigma_k}(m)\notin a$ and therefore that
\begin{equation}\label{June14ajkshd5}
d_{\sigma_k[g_k]}(n) = d_{\sigma_k}(n) \And
 d_{\sigma_k[g_k]}(m) = d_{\sigma_k}(m) .
\end{equation}
Moreover, if
$d_{\bar{\sigma}}(n) \notin a$ then
$$d_{\bar{\sigma}}(n) = d_{\sigma_{k+1}}(n) \leq i <
d_{\sigma_k}(n) = d_{\sigma_k[g_k]}(n)$$ and it follows from the fact that
$\bar{\sigma}\leq \sigma_k[g_k]$
that $$\sigma_{k+1}(n)(i) = \bar{\sigma}(n)(i) = \bar{\sigma}(m)(i)= \sigma_{k+1}(m)(i) .$$
A similar argument holds if
$d_{\bar{\sigma}}(m) \notin a$ and
therefore the remaining case is that 
\begin{itemize}
\item $d_{\bar{\sigma}}(n)\in a$
\item $d_{\bar{\sigma}}(m)\in a$ \item $ d_{\sigma_k}(n) = d_{\sigma_k}(m)\notin a$.
\end{itemize}
It follows that
\begin{itemize}
\item
$d_{\sigma_{k+1}}(n) =0 = d_{\sigma_{k+1}}(m)$ \item 
$\sigma_{k+1}(n) =0_n \cup \bar{\sigma}(n)$
\item 
$\sigma_{k+1}(m) =0_m \cup \bar{\sigma}(m)$
. 
\end{itemize}
Since Equation~(\ref{June14ajkshd5}) implies that
$d_{\sigma_k[g_k]}(n)  =  d_{\sigma_k[g_k]}(m) $ and $
\bar{\sigma} \leq \sigma_k[g_k]$
it follows that
$d_{\bar{\sigma}}(n) = d_{\bar{\sigma}}(m)$ and
if 
$d_{\bar{\sigma}}(n)\leq  i$ then 
$$\sigma_{k+1}(n)(i) = \bar{\sigma}(n)(i) = \bar{\sigma}(m)(i)
= \sigma_{k+1}(n)(i)$$
and so it can be assumed that
$d_{\bar{\sigma}}(n)> i$. But then 
$$\sigma_{k+1}(n)(i) = 0_n(i) = 0 = 0_m(i) = \sigma_{k+1}(m)(i) $$
as required.

Let $\vec{g}(i) = \sigma_M(i)\restriction N$ for each $i\in d_{\sigma_M}^{-1}\{0\}\setminus d_{\tau}^{-1}\{0\}$.
Define $\sigma$ by
$$\sigma(n) = \begin{cases}
\sigma_M(n)\restriction [N,n) & \text{ if } d_{\sigma_M}(n) = 0 \neq d_{\tau}(n) \\
\sigma_M(n) & \text{ otherwise}\\
\end{cases}$$
and note that $\sigma[\vec{g}] = \sigma_M\leq_a \tau$ (but it may not be the case that
$\sigma[\vec{g}]\leq \sigma$). 
Let $W=\{\sigma_{k+1}[g_k]\}_{k\in M}$ noting that $|W|=M\leq 2^{|A|\max(A)}$. A routine argument shows that $W$ is predense below $\sigma[\vec{g}]$.

To see that $\sigma\leq_A \tau$ first note that $\sigma^{-1}\{N\}\in \mathcal I$ since $\mathcal I$ is an ideal.
Since $\sigma_M\leq_a \tau$ it follows from the construction that
$\sigma_M^{-1}\{j\} = \tau^{-1}\{j\}$ for $j\in a$. Furthermore,
$\sigma_M^{-1}\{0\} = \tau^{-1}\{0\}$
by construction. 
Hence it only needs to be checked that
$\sigma\leq \tau$.
To this end, note that for $n\notin d_{\sigma_M}^{-1}\{0\}\setminus d_{\tau}^{-1}\{0\}$
the inclusions $\sigma(n) = \sigma_M(n)\supseteq \tau(n)$ hold since
$\sigma_M\leq \tau$.
 If $d_{\sigma_M}(n) = 0 =d_{ \tau}(n)$ the desired conclusion is immediate.
 If $d_{\sigma_M}(n) = 0$ and $d_{ \tau}(n)\neq 0$
and $d_{ \tau}(n) < N$ then it must be the case that
$d_{ \tau}(n)\in a$ and so
$\sigma_M(n) = \tau(n)$ because
$\sigma_M\leq_a \tau$. On the other hand, if 
$d_{ \tau}(n)\geq N$ then $\sigma(n) = 
\sigma_M(n)\restriction [N,n)\supseteq \tau(n)$ as required. The final point to note is that Condition~(\ref{BDragJG5}) of Definition~\ref{Sheas31ah7ib} follows from the fact that
$\sigma_M\leq\tau$.
\end{proof}

\begin{lemma}
	\label{29aa0957Drh75}
	If 	$\{\sigma_n\}_{n\in\omega}\subseteq  \Poset(\mathcal I)$ and
 \begin{itemize}
 \item
 $\{k_n\}_{n\in\omega}$ is a strictly increasing sequence of positive integers 
 \item $a_n=\{k_j\}_{j\in n}$  
 \item $a_n = \ran(d_{\sigma_n})\cap k_n$
	\item $\sigma_{n+1}\leq_{a_n} \sigma_n$ for each
	 $n\in \omega $
  \end{itemize}
and $\sigma$ is defined by $\sigma(m) = \bigcup_{n\in\omega}\sigma_n(m)$ then $\sigma\in \Poset(\mathcal I)$ and $\sigma\leq \sigma_n$ for each $n$.
Moreover, the conclusion holds
if it is only assumed that 
$a_n = \ran(d_{\sigma_n})\cap k_n\setminus \{0\}$
but \ $\bigcup_{n\in\omega}d_{\sigma_n}^{-1}\{0\}\in \mathcal I$.
\end{lemma}
\begin{proof} 
	Note that if $n<m$ then
	$$
	d_{\sigma_{n+1}}^{-1}\{k_n\}\cap d_{\sigma_{m+1}}^{-1}\{k_m\} = 
	d_{\sigma_{m+1}}^{-1}\{k_n\}\cap d_{\sigma_{m+1}}^{-1}\{k_m\}=\varnothing
	$$
	and if $j < k_n$ then
	$d_{\sigma_n}(j)\in a_n$.
	Hence
	 $\{d_{\sigma_{n+1}}^{-1}\{k_n\}\}_{n\in\omega}$
 is a partition of $\omega$ and
	therefore, if $d_\sigma$ is defined by
	$d_\sigma(j) = k_n$ if and only $ d_{\sigma_{n+1}}(j) = k_n$
	then $d_\sigma$ is a function defined for all $j$ and it witnesses that  Definition~\ref{Sheas31ah7ib} is satisfied by $\sigma$.
 The same argument proves the moreover clause. 
 \end{proof}

\begin{lemma}\label{Tu8777hDrh75}
	If 	$\mathcal I$ is a maximal ideal and $\tau\in \Poset(\mathcal I)$ and
	$D_n\subseteq \Poset(\mathcal I)$ are dense for $n\in\omega$
	 then there are 
 $\sigma\leq \tau$
 and  $\vec{g}_j :d_\sigma^{-1}\{j\}\to 2^j$ and  $W_j$ such that  for each $j$:
 \begin{enumerate}
\item\label{ItC1} $W_{2j}\cup W_{2j+1}\subseteq D_j$ 
\item \label{ItC2} $|W_{2j}|< 2^{(j-1)2j}$  and $|W_{2j+1}|< 2^{j(2j-1)}$ 
\item \label{ItC3} $W_{2j}\cup W_{2j+1}$ is predense below both
$\sigma\left[\bigcup_{i\leq j}\vec{g}_{2i}\right]$
and
$\sigma\left[\bigcup_{i< j}\vec{g}_{2i+1}\right]$
 \end{enumerate}
 and, furthermore, 
 either  $\sigma\left[\bigcup_{i\in \omega}\vec{g}_{2i}\right]\leq \tau$
 or $\sigma\left[\bigcup_{i\in \omega}\vec{g}_{2i+1}\right]\leq \tau$.
 Moreover, if $\mathfrak M$ is an elementary model of a sufficiently large fragment of set theory containing $\tau$ and $\mathcal I$ and such that $\{D_n\}_{n\in\omega}\subseteq \mathfrak M$ then it is possible to arrange that $W_j\in\mathfrak M$ for each $j$.
\end{lemma}
\begin{proof}
 It will first be shown how to construct $\{\sigma_n\}_{n\in\omega}$, $\vec{g}_j :d_\sigma^{-1}\{j\}\to 2^j$ and  $W_j$  for each $j$ such that:
 \begin{itemize}
\item $\sigma_{j+1}\leq_{j} \sigma_{j}$
\item $\vec{g}_j :d_{\sigma_{j+1}}^{-1}\{j\}\to 2^j$
\item conditions~(\ref{ItC1}) and (\ref{ItC2}) hold
\item 
$\sigma_{2j}\left[\bigcup_{i\in  j}
 \vec{g}_{2i+1}\right]\leq \tau$
\item 
$\sigma_{2j+1}\left[\bigcup_{i\leq  j}
\vec{g}_{2i}\right]\leq \tau$
\item $W_{2j}$ is predense below
$\sigma_{2j}\left[\bigcup_{i\in  j}
\vec{g}_{2i+1}\right]$ 
\item $W_{2j+1}$ is predense below $\sigma_{2j+1}\left[\bigcup_{i\leq  j}
\vec{g}_{2i}\right]$.
 \end{itemize}
	If this can be done then
	let $\sigma(n) = \bigcup_{j\in\omega}\sigma_j(n)$ and  use
	 Lemma~\ref{29aa0957Drh75} to conclude that $\sigma  \in \Poset(\mathcal I)$ and $\sigma\leq \tau$.
	Then the sets $d_{\sigma_{j+1}}^{-1}\{j\}$ are pairwise disjoint and so the maximality of $\mathcal I$ implies that
 at least one of the following holds: 
 \begin{enumerate}
\item \label{Oneqw} $\bigcup_{j\in\omega}d_{\sigma_{2j+1}}^{-1}\{2j\}\in \mathcal I$
\item \label{Oneqw2} $
 \bigcup_{j\in\omega}d_{\sigma_{2j+2}}^{-1}\{2j+1\}\in \mathcal I .$
 \end{enumerate}
If (\ref{Oneqw}) holds then
 $\sigma\left[\bigcup_{i\in \omega}\vec{g}_{2i}\right]
 \in \Poset(\mathcal I)$
 and $\sigma\left[\bigcup_{i\in \omega}\vec{g}_{2i}\right]\leq \tau$ while if
(\ref{Oneqw2}) holds then
 $\sigma\left[\bigcup_{i\in \omega}\vec{g}_{2i+1}\right]
 \in \Poset(\mathcal I)$
and $\sigma\left[\bigcup_{i\in \omega}\vec{g}_{2i+1}\right]\leq \tau$.
Conditions~(\ref{ItC1}) and (\ref{ItC2}) hold by construction. Conditions~(\ref{ItC3})  holds because
$W_{2j}\cup W_{2j+1}$ is predense
below both 
$\sigma_{2j}\left[\bigcup_{i\in  j}
\vec{g}_{2i+1}\right]$ 
and $\sigma_{2j+1}\left[\bigcup_{i\leq  j}
\vec{g}_{2i}\right]$ and
$$\sigma\left[\bigcup_{i\in \omega}\vec{g}_{2i}\right]\leq 
\sigma_{2j+1}\left[\bigcup_{i\leq  j}
\vec{g}_{2i}\right]
$$
and
$$\sigma\left[\bigcup_{i\in \omega}\vec{g}_{2i+1}\right]\leq \sigma_{2j}\left[\bigcup_{i\in  j}
\vec{g}_{2i+1}\right] .$$

 To carry out the induction suppose
 that
 $\sigma_{2n}$ has been constructed --- to prove the {\em moreover} clause assume that the first $2n$ steps of the inductive construction have been carried out in the model $\mathfrak M$. Then 
 use Lemma~\ref{2a9Tu8777hDrh75}
  with 
 $\tau = \sigma_{2n}\left[\bigcup_{i\in n}\vec{g}_{2i+1}\right]$ and
 $N=2n$ and noting that 
 $\ran(d_\tau) = \{2i+1\}_{i \in n}=a$
  to find
  $\sigma^*_{2n+1}\leq_a \sigma_{2n}\left[\bigcup_{i\in n}\vec{g}_{2i+1}\right]$ and $W_{2n+1}\subseteq D_n$ 
  and
  $$\vec{g}_{2n}:d_{\sigma^*_{2n+1}}^{-1}\{2n\}\to 2^{2n}$$ satisfying the conclusion of Lemma~\ref{2a9Tu8777hDrh75}
  and note that, since $D_n\in \mathfrak M$, elementarity guarantees that $W_{2n+1}$ and $\sigma^*_{2n-1}[\vec{g}_{2n}]$ can be chosen in $\mathfrak M$.
   In other words, the following three conditions hold:
  \begin{equation}\label{Cndh6hAlDz1}
  	|W_{2n+1}|\leq 2^{n(2n-1)}
  	\end{equation}
  \begin{equation}\label{1Cndh6hAlDz1}
  	W_{2n+1}\text{ is predense below }
  	\sigma^*_{2n-1}[\vec{g}_{2n}] .
  \end{equation}
 \begin{equation}\label{2Cndh6hAlDz1}
 	W_{2n+1}\in \mathfrak M
 	 \end{equation}
 	 with condition (\ref{2Cndh6hAlDz1}) being relevant only in the case of the {\em moreover } clause.
  
  Then define $\sigma_{2n+1}$ by
  $$\sigma_{2n+1}(\ell) = \begin{cases}
\sigma^*_{2n+1}(\ell) &\text{ if } d_{\sigma^*_{2n+1}}(\ell)\geq 2n\\
\sigma_{2n}(\ell) &\text{ if } d_{\sigma^*_{2n+1}}(\ell)< 2n .
  \end{cases}$$
 It follows that $\sigma_{2n+1}\leq_{2n} \sigma_{2n}$ as required. A similar argument now yields $\sigma_{2n+2}$, $W_{2n+2}$ and $\vec{g}_{2n+2}$.
\end{proof}

\begin{defin}
	A partial order $\Poset$ has the {\em Sacks Property} (see Definition~2.9A on page 291 of \cite{MR1623206} or Definition~6.3.37 on page 303 of \cite{MR1350295}) if whenever
	$p \forces{\Poset}{\dot{f}:\omega \to \omega}$ and $g:\omega\to\omega$ and $\lim_{n\to\infty}g(n) = \infty$ there is $q\leq p$
	and $F\in \prod_n[\omega]^{g(n)}$ such that
	$q\forces{\Poset}{(\forall n\in \omega) \ \dot{f}(n) \in {F}(n)}$.
	A partial order $\Poset$ has the {\em Weak Sacks Property} (see Definition~2.12A on page 298 of \cite{MR1623206} where it is called the {\em Strong PP Property}) if whenever
	$p \forces{\Poset}{\dot{f}:\omega \to \omega}$ and $g:\omega\to\omega$ and $\lim_{n\to\infty}g(n) = \infty$ there is $q\leq p$
	and $F\in \prod_n[\omega]^{g(n)}$ and an infinite $A\subseteq \omega $ such that
	$q\forces{\Poset}{(\forall n\in \check{A}) \ \dot{f}(n) \in F(n)}$.
\end{defin}
The following theorems were proved by Shelah as part of a large collection of preservation lemmas in 
\S~2.12A of \cite{MR1623206}.
\begin{theor}\label{ShePrewws}
	If $\Poset_\beta$ is proper and has the Weak Sacks Property for each $\beta\in \alpha$ and $\Poset_\alpha$
	is the countable support iteration of the $\Poset_\beta$ then $\Poset_\alpha$ also has the Weak Sacks Property. 
	\end{theor}
	\begin{theor}\label{ShePrewwsA}
		If  $\Poset$ has the Weak Sacks Property and
  \begin{align*}
   {\sf 1}\forces{\Poset}{\mathbb Q\text{ has the Weak Sacks Property}}
  \end{align*}
   then 
		$\Poset\ast\mathbb Q$ has the Weak Sacks Property. 
	\end{theor}
\begin{corol}	\label{CCOOR777hDrh75}
	If $\mathcal I$ is a maximal ideal then 	$\Poset(\mathcal I)$ has the Sacks Property.
	\end{corol}
	\begin{proof}
		Suppose that $\tau\forces{\Poset(\mathcal I)}{\dot{f}:\omega\to\omega}$ and $\lim_{n\to\infty}g(n) =\infty$. 
  Let  $L_n$ be so large that $g(\ell)> 2^{2n^2}$ for all $\ell\geq L_n$. 
Then let $D_n$ be the dense set of conditions $\theta\in \Poset({\mathcal I})$
such that there is  $h_\theta:[L_n,L_{n+1})\to 2$  such that 
$$\theta\forces{\Poset({\mathcal I})}{\dot{f}\restriction[L_n,L_{n+1}) = h_\theta} . $$
	Let 
$W_j$, $\vec{g}_j$ and $\sigma^*$ be as guaranteed by Lemma~\ref{Tu8777hDrh75} for $\{D_j\}_{j\in\omega}$  and $\tau$ and suppose, without loss of generality, that
$\sigma\left[\bigcup_{i\in \omega}\vec{g}_{2i}\right]\leq \tau$. Then $W_{2n}\subseteq D_n$ and is predense
below $\sigma\left[\bigcup_{i\in \omega}\vec{g}_{2i}\right]$.

Define $F:\omega\to [\omega]^{<\aleph_0}$ by
$$F(j) = \SetOf{h_\theta(j)}{L_n\leq j < L_{n+1} \ \And \ \theta\in W_{2n}   }$$
and note that if $L_n\leq j < L_{n+1}$ then 
$$|F(j)|\leq |W_{2n}|\leq 2^{2n^2} < g(j) .$$
Since $W_{2n}$ is predense
below $\sigma\left[\bigcup_{i\in \omega}\vec{g}_{2i}\right]$
it follows that
 $$\sigma\left[\bigcup_{i\in \omega}\vec{g}_{2i}\right]\forces{\Poset({\mathcal I})}{\dot{f}(j)\in F(j)}$$
 if $L_n\leq j < L_{n+1}$ and, hence,
 $\sigma\forces{\Poset({\mathcal I})}{(\forall j ) \ \dot{f}(j)\in F(j)}$.
\end{proof}

\begin{corol}	\label{CCOOR777hDrh75A}
	If $\mathcal I$ is a maximal ideal then $\Poset(\mathcal I)$ is proper.
\end{corol}
\begin{proof}
	Let $\mathfrak M$ be a countable elementary submodel of $\mathsf{H}({\kappa}^{+})$ such that $\mathcal I\in \mathfrak M$ and let $\tau\in \mathfrak M\cap \Poset(\mathcal I)$.
	Let $\{D^*_n\}_{n\in\omega}$ enumerate all the dense open subsets of $\Poset(\mathcal I)$ in $\mathfrak M$ and
	let $D_n =\bigcap_{j\leq n}D^*_j$.  Then use
	Lemma~\ref{Tu8777hDrh75} as in Corollary~\ref{CCOOR777hDrh75}  to find
	$\sigma\leq \tau$ such that for infinitely many  $j$ there is  a finite set $W_j\subseteq D_j$ such that 
	 $W_j$ is predense below  $\sigma$. 
	 Using the {\em moreover} clause of Lemma~\ref{Tu8777hDrh75}
	  it can be assumed that the $W_j$ are all subsets of $\mathfrak M$ and, therefore, $\sigma$ is $\mathfrak M$ generic.
	
\end{proof}
\subsection{Application to nowhere dense ultrafilters}
This section shows how to use the partial orders $\Poset(\mathcal I)$ to construct a model with no nowhere dense ultrafilters.	The following lemma is well known.
	\begin{lemma}\label{LabiSInznsnow}
		If $Z\subseteq 2^\omega$ is nowhere dense 
  then for each $k\in \omega$ there is some $\psi(k,Z)\in \omega$
  and
		$t:[k,\psi(k,Z))\to 2$ such that if $z\in Z$ then $t\not\subseteq z$.
	\end{lemma}
\begin{proof}  Let
$\{s_j\}_{j=1}^M$ enumerate $2^{k}$ and choose inductively
$t^*_j:[k, \ell_j)\to 2$ such that the open neighbourhood determined by $s_j\cup t^*_j$ is disjoint from $Z$ and $t^*_j\subseteq t^*_{j+1}$. Let $\psi(k,Z)= \ell_M$ and
$t= t^*_M$.
\end{proof}

For the purposes of this section it would suffice to prove the following
lemma only for partial orders $\Poset$ satisfying the Sacks Property. However, in \S\ref{VarsrSect} it will be useful to have already established the following general lemma. 
	\begin{lemma}\label{LabiSInznsCOFnanow}
	If $\Poset$ has the Weak Sacks Property and
	 ${\sf 1}\forces{\Poset}{\dot{Z}\subseteq 2^\omega\text{ is nowhere dense}}$
then there is $p\in \Poset$, and an infinite $X\subseteq \omega$ such that for each $n\in X$
there are $N(n)$ and $S_n:[n,N(n))\to 2$ such that
	$
		p\forces{\Poset}{(\forall z\in \dot{Z}) \ S_n\not\subseteq z} $.
\end{lemma}
\begin{proof}
From Corollary~\ref{CCOOR777hDrh75} it follows  that $\Poset$ is $\omega^\omega$-bounding and, hence,
there is a function $\Psi$ such that
$${\sf 1}\forces{\Poset}{(\forall n) \ \Psi(n)\geq \psi(n,\dot{Z})} .$$
where $\psi$ is the function in Lemma~\ref{LabiSInznsnow}.
It is then possible to find names $\dot{t}_{n,i}$ for $i\in n$ such that
letting $N_0(n)=n$ and $N_{j+1}(n) = \Psi(N_j(n))$
$${\sf 1}\forces{\Poset}{(\forall i \in n\in \omega) \ \dot{t}_{n,i}:[N_i(n),N_{i+1}(n))\to 2 \And
(\forall z\in \dot{Z}) \ \dot{t}_{n,i}\not\subseteq z} .$$
Use the Weak Sacks Property of $\Poset$ to find $p\in \Poset$, and an infinite $X\subseteq \omega$
and  $W_n$ for $n\in X$ such that
 $|W_n|=n$ and
 $p\forces{\Poset}{\bigcup_{i\in n} \dot{t}_{n,i}\in W_n}$ for each $n\in X$.
 
 	Let $W_n=\{s_j\}_{j\in n}$ and let $S_n = \bigcup_{i\in n}s_i\restriction [N_i(n),N_{i+1}(n))$. 
	 Letting $N(n) = N_n(n)$ for $n\in X$
  it follows that 
  $S_n:[n,N(n))\to 2$ satisfies the lemma. 
\end{proof}
\begin{defin}
	If $G\subseteq \Poset(\mathcal {I})$ is generic define
	$F_G:\omega\to 2^\omega$ by
	$$F_G(j) =\begin{cases}
		\left(\bigcup_{\sigma\in G}\sigma(n)\right)(j)&\text{ if } j < n\\
		0 &\text{ if } j\geq n.
	\end{cases}$$
\end{defin}	
	\begin{theor}
		\label{NoNowjlhMyV}
	If $\mathcal I$ is a maximal ideal	
	and 
	$\Poset(\mathcal I)\ast \mathbb Q$ has the  Weak Sacks Property
	then for any
	$(\tau,p)$ such that 
	 $$(\tau,p)\forces{\Poset(\mathcal I)\ast \mathbb Q}{\dot{Z}\subseteq 2^\omega\text{ is nowhere dense}}$$   there are $(\sigma, q)\leq (\tau, p) $ and $U\in \mathcal I^*$ such that
		$(\sigma,q)\forces{\Poset(\mathcal I)\ast \mathbb Q}{F_{\dot{G}}({U})\cap \dot{Z}=\varnothing}$.
	\end{theor}
\begin{proof}
	Since $\Poset(\mathcal I)\ast \mathbb Q$ has the Weak Sacks Property it is possible to appeal to  Lemma~\ref{LabiSInznsCOFnanow}  to find
	$(\tau^*,q)\leq (\tau,p)$, an infinite $X\subseteq \omega$ and	$S_n:[n,{N(n)})\to 2$   such that
	\begin{equation}\label{Juj21bmm}
	(\tau^*,q)\forces{\Poset(\mathcal I)\ast\mathbb Q}{(\forall z\in \dot{Z})(\forall x\in {X}) \ S_x\not\subseteq z} .
	\end{equation}
Let $\{x_i\}_{i\in\omega}$ enumerate $X\setminus \{0\}$ in increasing order.
Define the function  $e$  by
$$e(j) = 
\begin{cases}
0 &\text{ if } j<N(x_0)\\
x_{i}
	&\text{ if } N(x_i)\leq j < N(x_{i+1}) 
\end{cases}$$
noting that $e(j)\leq j$.
Then define $\sigma$
such that $\sigma(j):[e(d_{\tau^*}(j)),j)\to 2$ by
 $$\sigma(j)(\ell) =
\begin{cases}
\tau^*(j)(\ell) &\text{ if } \ell \geq d_{\tau^*}(j))\\
S_{x_i}(\ell) &
\text{   if } x_i\leq \ell <N(x_i)\leq d_{\tau^*}(j) < N(x_{i+1}) \\
 0 &
\text{   if } N(x_i)\leq \ell < d_{\tau^*}(j)\\
0 &
\text{   if }  \ell < d_{\tau^*}(j) < N(x_0).
\end{cases}$$
Then $d_\sigma  = e\circ d_{\tau^*}$ and
 the  domain of each function 
$\sigma(j)$ is the interval $[d_\sigma(j),j)$ as required to belong to $\Poset(\mathcal I)$. Moreover, $\sigma(j)\supseteq \tau^*(j)$ for each $j\in\omega$.
 The key point to notice here is that Condition~(\ref{BDragJG5}) of Definition~\ref{Sheas31ah7ib} is satisfied. 
 To see this suppose that  $d_{\tau^*}(m) = d_{\tau^*}(n)$ and
 $e(d_{\tau^*}(m) ) = i$ then $\sigma(m)\setminus \tau^*(m)$ and
$\sigma(n)\setminus \tau^*(n)$ have the same domain and
$$S_{x_i}^{-1}\{1\} = \left(\sigma(m)\setminus \tau^*(m)\right)^{-1}\{1\} = \left(\sigma(n)\setminus \tau^*(n)\right)^{-1}\{1\} .
$$
 Hence  $\sigma\leq \tau^*$.
 
Let $U= \omega \setminus  d_{\sigma}^{-1}(0) $ and note that
$U\in \mathcal I^*$ and $$(\sigma,q)\forces{\Poset(\mathcal I)\ast\mathbb Q}{(\forall n\in U)(\exists m\in\omega) \ S_{x_m}\subseteq F_{\dot{G}}(n)}$$ and hence (\ref{Juj21bmm}) implies that
$(\sigma,q)\forces{\Poset(\mathcal I)\ast\mathbb Q}{(\forall n\in U) \  F_{\dot{G}}(n)\notin \dot{Z}}$
as required. 
\end{proof}

	\begin{corol}
	\label{CorNoNowjlhMyV}
	It is consistent that there are no nowhere dense ultrafilters.
\end{corol}
\begin{proof}
	Let $V$ be a model of set theory such that $\lozenge_{\omega_2,\omega_1}$ holds and this is witnessed by $\{D_\alpha\}_{\alpha\in\omega_2} $ --- in other words, $D_\alpha\subseteq \alpha$ and for every $X\subseteq \omega_2$
 $$\SetOf{\alpha\in \omega _2}{ \cof(\alpha) = \omega_1 \And X\cap\alpha =  D_\alpha}$$
 is stationary. 
Construct a countable support iteration	 $\{\Poset_\alpha\}_{\alpha\in\omega_2}$ such that for each $\alpha\in \omega_2$ if $D_\alpha$
	 is a $\Poset_\alpha$ name such that ${\sf 1}\forces{\Poset_\alpha}{D_\alpha\text{ is a maximal ideal}}$ then
	$\Poset_{\alpha+1} = \Poset_\alpha*\Poset(D_\alpha)$. Since the iteration is proper  and $|\Poset_\alpha|=\aleph_1$ for each $\alpha\in \omega_2$ it follows that if $\dot{\mathcal U}$ is a $\Poset_{\omega_2}$ name for an ultrafilter then the set of $\alpha\in \omega_2$ such that for any generic $G\subseteq \Poset_{\omega_2}$ the interpretation of 
	$D_\alpha$ in $V[G\cap \Poset_\alpha]$ is equal to the interpretation of $  \dot{\mathcal U}^*$ in $V[G\cap \Poset_\alpha]$ contains a closed unbounded set. 
	
	Since the Sacks Property is preserved by countable support iterations (see \S2.12A of Chapter~VI of \cite{MR1623206} or Application~5 on page 351 of
	\cite{MR1234283}) it follows that in the model $V[G\cap \Poset_\alpha]$ the partial order
	$\Poset(D_\alpha)\ast\Poset_{\omega_2}/\Poset_{\alpha+1}$  has the Sacks Property.
	 By applying Theorem~\ref{NoNowjlhMyV} in the model $V[G\cap \Poset_\alpha]$  it follows  that  $\Poset(D_\alpha)\ast\Poset_{\omega_2}/\Poset_{\alpha+1}$ forces the existence  of a function $F:\omega \to 2^\omega$ such that if
${F}(U)$ is nowhere dense then $U\in D_\alpha$. Hence
if ${\sf 1}\forces{\Poset_{\omega_2}}{\dot{\mathcal U}\text{ is an ultrafilter}}$ then 
$${\sf 1}\forces{\Poset_{\omega_2}}{(\exists F:\omega \to 2^\omega)(\forall Z\subseteq 2^\omega) \ \text{ if $Z$ is nowhere dense then } F^{-1}(Z)\in D_\alpha\subseteq \dot{\mathcal U}^*}$$
and so
${\sf 1}\forces{\Poset_{\omega_2}}{\text{there are no nowhere dense ultrafilters} }$.
	
\end{proof}
\subsection{Comparison with Shelah's original partial order}
\label{renaj6} It has already been remarked that the partial order of \cite{MR1690694}, which will be referred to as $\Poset_{\sf Shelah}(\mathcal I)$, does not enjoy the Sacks Property. It is the goal of this section to explain this remark. A purely superficial difference between
$\Poset_{\sf Shelah}(\mathcal I)$ and $\Poset(\mathcal I)$
of Definition~\ref{Sheas31ah7ib} is that a condition $\sigma\in \Poset_{\sf Shelah}(\mathcal I)$ is defined by using equivalence classes which would correspond to the fibres of the function $d_\sigma$ used to define $\Poset(\mathcal I)$. A more substantial difference is that if the definition of $\Poset_{\sf Shelah}(\mathcal I)$ were to be restated using the functions $d_\sigma$ instead of equivalence classes, then the following extra assumption would need to be added 
\begin{equation}\label{Extrfagfs45}d_\sigma(j) = \min((d_\sigma^{-1}(d_\sigma(j))) .
	\end{equation} In \cite{MR1690694} this extra requirement is used to establish homogeneity properties of the partial order which are not needed for the central argument, but may be useful in other contexts.

However, this extra assumption does destroy the Sacks Property. To see this, let $\dot{G}$ be a $\Poset_{\sf Shelah}(\mathcal I)$-name for the generic function obtained after forcing with
$\Poset_{\sf Shelah}(\mathcal I)$ or, in other words, forcing with the subset of conditions in $\Poset(\mathcal I)$
that satisfy Equation~(\ref{Extrfagfs45}).
Keep in mind that $\dot{G}$ is forced to be an element of $ \prod_{n\in\omega}2^n$.   Letting $g(n)=n$
it will be shown that given any $\sigma\in \Poset_{\sf Shelah}(\mathcal I)$ and any function 
$$F\in \prod_{n\in\omega}[2^n]^n$$
there is $\tau \leq \sigma$ and $j$ such that
\begin{equation}\label{Notqashw516}
	\tau\forces{\Poset_{\sf Shelah}(\mathcal I)}{\dot{G}(j)\notin F(j)} .\end{equation}
To find $\tau$ simply let $j$ be such that
$j=\min(d_\sigma^{-1}(d_\sigma(j)))>1$.
Then $|F(j)| = g(j) = j< 2^j$ and so there is some $f\in 2^j\setminus F(j)$. Defining $\tau$ by
$$\tau(\ell) = \begin{cases}
	\sigma(\ell) &\text{if } d_\sigma(\ell) \neq j\\
	f\cup \sigma(\ell) &\text{if } d_\sigma(\ell) =j
	\end{cases}$$
	it is easy to check that $\tau\leq \sigma$ and that (\ref{Notqashw516}) holds.

There are several other naturally occurring forcing notions that have the Weak Sacks Property, but not the Sacks Property.
Indeed, the forcing notion which is about to be presented in \S\ref{VarsrSect} is one of them.
Another example occurs in \S8 of \cite{souvssu}, where Definition 8.2 defines a forcing that has the Weak Sacks Property, yet not the Sacks Property and does not add an independent real, yet does not preserve P-points.
See the remarks preceding Lemma 8.13 of \cite{souvssu}.
\section{Variations on selectivity}\label{VarsrSect}
\subsection{Definition of $(m,n)$-selectivity}
The following two definitions and notation were introduced in \cite{MR3946661} to produce some examples in the theory of topological dynamics. The goal of this section is to modify the construction of \S\ref{Ex} in such a way that the resulting model contains $(m,n)$-selective ultrafilters, but still no nowhere dense ultrafilters.
\begin{defin}\label{BartZuck}
A subset $A\subseteq [\omega]^m$ is defined to be {\em  thick }
if and only if for every $n\geq m$, there is $s\in [\omega]^n$
such that $[s]^m\subseteq A$.
 A filter will be said to be \emph{thick} if each of its members is thick.
\end{defin}
\begin{defin}\label{BartZuck2}
	If $A\subseteq [\omega]^m$ and $n\geq m$ define
	$\lambda^{n-m}(A) = \SetOf{s\in [\omega]^n}{[s]^m\subseteq A}$.
	If $\mathcal F$ is a filter on $[\omega]^m$ then $\lambda^{n-m}(\mathcal F)$ is defined to be the filter generated by 
	$\SetOf{\lambda^{n-m}(A)}{A\in\mathcal F}$.
	\end{defin}
Note that if  $A\subseteq [\omega]^m$ is thick
then $\lambda^{n-m}(A)\neq \varnothing$ but that, without this assumption it can happen that
$\lambda^{n-m}(A) =  \varnothing$. Hence, the definition of $\lambda^{n-m}(\mathcal F)$ is only of interest for thick filters.
\begin{defin}\label{BartZuck3}
A thick ultrafilter $\mathcal U$ on $[\omega]^m$ is defined to be \emph{$(m,n)$-selective} if and only if ${\lambda}^{n-m}(\mathcal{U})$ is an ultrafilter.
\end{defin}
Observe that if $k \leq m \leq n$ and $\mathcal{U}$ is $(k, n)$-selective, then any ultrafilter on ${[\omega]}^{m}$ that extends ${\lambda}^{m-k}(\mathcal{U})$ will be $(m, n)$-selective.
Observe also that selective ultrafilters, when considered as ultrafilters on ${[\omega]}^{1}$, are $(1,n)$-selective for all $n \geq 1$, but it is shown in \cite{MR3946661} that $(2,3)$-selective ultrafilters may not be of the form ${\lambda}^{2-1}(\mathcal{U})$, for any selective $\mathcal{U}$, assuming   $2^{\aleph_0} = \aleph_1$.
However, Question~7.7 of \cite{MR3946661} asks whether
$(2,3)$-selective ultrafilters can be shown to exist without assuming any extra set theoretic axioms. 
Note also that by a theorem of Kunen (see Theorem~4.5.2 of \cite{MR1350295}) it follows that
$(1,2)$-selective ultrafilters are already selective. 

An attempt at providing a negative answer will, among other things, require producing a model of set theory without selective ultrafilters. 
 It has already been noted in \S\ref{ievads} that several such models are known. 
Among these models the one that seems to eliminate the most ultrafilters with peculiar properties seems to the model of \cite{MR1690694} 
described in  \S\ref{Ex}.
Indeed, the optimist may even conjecture that the existence of a \((2,3)\)-selective ultrafilter implies the existence of a nowhere dense one and, hence, there are no $(2,3)$-selective ultrafilters if there are no nowhere dense ones.
It should be noted, however, that a result of Barto\v{s}ov\'{a} and Zucker from \cite{MR3946661} shows that to prove that all $(2, 3)$-selectives are nowhere dense it is necessary to assume  some extra set-theoretic axioms.
The following is a corollary to the proof of Theorem~7.2 in \cite{MR3946661}.
\begin{Theor}[Barto\v{s}ov\'{a} and Zucker~\cite{MR3946661}] \label{thm:BZ}
 If $2^{\aleph_0} = \aleph_1$
 then for every ultrafilter \(\mathcal{U}\), there exists a $(2, 3)$-selective ultrafilter $\mathcal{V}$ such that $\mathcal{U} \; {\leq}_{RK} \; \mathcal{V}$.
\end{Theor}
Theorem \ref{thm:BZ} has a consequence that is worth explicitly pointing out: there is no non-trivial ideal $\mathcal{I}$ for which it is possible to prove that all $(2, 3)$-selectives are $\mathcal{I}$-ultrafilters without assuming some extra set-theoretic axioms.
In particular, $(2, 3)$-selectives are not provably nowhere dense.
\begin{corol} \label{cor:BZ}
  If $2^{\aleph_0} = \aleph_1$
  and $\mathcal{I}$ is a non-principal ideal on the set $X$ such that ${\left[X\right]}^{{\aleph}_{0}} \not\subseteq \mathcal{I}$ then there exists a $(2, 3)$-selective ultrafilter that is not an $\mathcal{I}$-ultrafilter.
\end{corol}
\begin{proof}
 By hypothesis, there is a countably infinite set $Y \subseteq X$ with $Y \notin \mathcal{I}$.
 Fix a bijection $f: \omega \rightarrow Y$.
 Define $\mathcal{F} = \{{f}^{-1}(Y \setminus A): A \in \mathcal{I}\}$.
 Then $\mathcal{F}$ is a non-principal proper filter on $\omega$ because $\mathcal{I}$ is non-principal and $Y \notin \mathcal{I}$.
 Extend $\mathcal{F}$ to an ultrafilter $\mathcal{U}$ and apply Theorem \ref{thm:BZ} to find a $(2, 3)$-selective $\mathcal{V}$ such that $\mathcal{U} \; {\leq}_{RK} \; \mathcal{V}$, witnessed by some $g: \omega \rightarrow \omega$.
 Then $f \circ g: \omega \rightarrow X$ witnesses that $\mathcal{V}$ is not an $\mathcal{I}$-ultrafilter.
\end{proof}
Obviously, Theorem \ref{thm:BZ} and its Corollary \ref{cor:BZ} are silent on the possibility that nowhere dense ultrafilters, or even P-points, might exist whenever $(2, 3)$-selective ultrafilters do.
It is the purpose of this section to show that the existence of $(2, 3)$-selectives does not imply the existence of a nowhere dense ultrafilter and that, even more, there is good reason to believe that the model of
\cite{MR1690694} contains $(2,3)$-selectives.

This statement requires some explanation. As has been seen in  \S\ref{Ex}, the model of
\cite{MR1690694} is a countable support iteration of partial orders designed to introduce functions that will prevent any ultrafilter from becoming nowhere dense. The partial orders consist of sequences of sets from the dual ideal with some additional structure. 
Each of the iterands is proper, $\omega^\omega$-bounding and has the  Sacks Property.
While it will not be shown in this note that there are (2,3)-selectives in this specific model, it will be shown that it is possible to interleave the iteration with partial orders that are
proper, $\omega^\omega$-bounding and have the Weak Sacks Property in such a way that the resulting model does have a (2,3)-selective ultrafilter.
Indeed, it will be shown that there is an ultrafilter
on $[\omega]^2$ that is a $(2,d)$-selective ultrafilter for all $d\geq 2$ simultaneously.
It is easy to see that if $\mathcal{U}$ is a $(2, m)$-selective ultrafilter for each $2 \leq m \leq d$, then ${\lambda}^{m - 2}(\mathcal{U})$ is an $(m, d)$-selective ultrafilter.
Note that this is as strong a version of selectivity as one might expect in a model with no nowhere dense ultrafilters because a $(1,2)$-selective ultrafilter is a selective ultrafilter. 
 \begin{defin}
  A thick ultrafilter on $[\omega]^2$ that is a $(2,d)$-selective ultrafilter for all $d\geq 2$ simultaneously will be called a \emph{$(2,\aleph_0)$-selective} ultrafilter.
 	\end{defin}
\subsection{A general partial order}
\label{SSag}
The modification of the model of \cite{MR1690694} will 
require constructing a $(2,\aleph_0)$-selective ultrafilter by an iteration of length $\omega_2$. At each stage of the construction a thick set $Y\subseteq [\omega]^2$ will be added so that for a given partition of
$[\omega]^d= A_0 \cup A_1$ either $\lambda^{d-2}(Y)\subseteq A_0$ or 
$\lambda^{d-2}(Y)\subseteq A_1$. 
The partial orders for doing this will be special instances of a more general partial order described in this subsection.

\begin{theor}[Jalali-Naini, Talagrand \cite{MR0579439, JN}]\label{NajMTal}  An ideal $\mathcal I$ on $\omega$ is non-meagre if and only if for every
increasing sequence of integers $\{n_i\}_{i\in\omega}$ there is an infinite $X\subseteq \omega$ such that $\bigcup_{i\in x}[n_i, n_{i+1})\in \mathcal I$.
\end{theor}
\begin{corol}\label{Fhvcttar5}
If $\mathcal I$ is a non-meagre ideal on $\omega$ and $\Poset$ is an $\omega^\omega$ bounding partial order then 
${\sf 1}\forces{\Poset}{\mathcal I \text{ generates a non-meagre ideal}}$.
\end{corol}

\begin{defin}
Given a filter $\mathcal F$ on a countable set $Z$ say that {\em {\sf Player 1} has a winning strategy in the game $\Game(\mathcal F)$} if there is a tree
$\Sigma\subseteq \left([Z]^{<\aleph_0}\right)^{<\omega}$ such that
\begin{itemize}
\item for each $\sigma \in \Sigma$ there is $A_\sigma\in \mathcal F$ such that
$[A_\sigma]^{<\aleph_0}\subseteq \suc_\Sigma(\sigma)$
\item $\bigcup_{n\in\omega}\left(\bigcup B\right)(n)\notin \mathcal F$ for every branch $B\subseteq \Sigma$.
\end{itemize}
And, of course, if it fails to be the case that {\sf Player 1} has a winning strategy in the game $\Game(\mathcal F)$ then it will be said that {\em 
{\sf Player 1} has no winning strategy in the game $\Game(\mathcal F)$.}
\end{defin}
\begin{lemma}\label{Hjhg65y}
If $\mathcal F$ is a non-meagre, P-filter then
{\sf Player 1} has no winning strategy in the game $\Game(\mathcal F)$.
\end{lemma}
\begin{proof}
Use Theorem~\ref{NajMTal} and the standard proof of the game characterization of P-points as, for example, in Theorem~4.4.4 of \cite{MR1350295}.
\end{proof}

\begin{defin}\label{Dfif9tqqehgPo}
If
 $\mathcal V$ is a filter on $\omega$ and
 ${\bf T}\subseteq Z^{<\omega}$ is a tree
 define $\Poset(\mathcal V, {\bf T})$ to consist of trees $T\subseteq {\bf T}$ such that
 $$\SetOf{n\in\omega}{
(\forall t\in T) \ \text{ if }|t| = n \text{ then }
\textstyle{\suc_T(t) = \suc_{\bf T}(t)}
}\in \mathcal V$$
where $\suc_T(t) = \SetOf{x}{t^\frown x\in T}$.
The ordering on $\Poset(\mathcal V, P)$ is $\subseteq$.
 \end{defin}

 \begin{notat}
For a tree $T$ and $k\in \omega$ let $T\restriction k = \SetOf{t\in T}{|t| = k}$ and for $t\in T$ let $T\angbr{t} = \SetOf{s\in T}{s\subseteq t \text{ or } t\subseteq s}$.
 \end{notat}

\begin{lemma}\label{Lakk6hy7}
Suppose that $\mathcal V$ is a non-meagre P-filter on $\omega$ and ${\bf T}\subseteq Z^{<\omega}$ is a tree with no terminal nodes.
Then $\Poset(\mathcal V, {\bf T})$ of Definition~\ref{Dfif9tqqehgPo} is proper and has the Weak Sacks Property.
\end{lemma}
\begin{proof}
The first step in showing that $\Poset(\mathcal V, {\bf T})$ is proper and has the Weak Sacks Property is to establish the following:
\begin{claim}
Suppose that $T\in \Poset(\mathcal V, {\bf T})$, $\mathfrak M$ is
a countable,  elementary submodel of $\mathsf{H}({\kappa}^{+})$ for some uncountable $\kappa$ and that
$\{D_n\}_{n\in\omega}$ is an enumeration some dense subsets of $\Poset(\mathcal V, {\bf T})\cap\mathfrak M$. Then there is $T^*\subseteq T$ such that there are infinitely many $n\in\omega$
for which there is a family
$\{\bar{T}_t\}_{t\in T^*\restriction n}$ such that
for all $t\in T^*\restriction n$:
\begin{itemize}
\item $\bar{T}_t \in D_n\cap  \mathfrak M$
 \item $\bar{T}_t\subseteq T\angbr{t}$
 \end{itemize}
 and 
 $T^*\subseteq \bigcup_{t\in T^*\restriction n} \bar{T}_t$.
\end{claim}
\begin{proof}
It is routine to construct a tree $\Sigma\subseteq \left([Z]^{<\aleph_0}\right)^{<\omega}$ such that for each $\sigma\in \Sigma$ there are $A_\sigma \in \mathcal V\cap \mathfrak M$ and $T_\sigma\in \Poset(\mathcal V, {\bf T})\cap \mathfrak M$ such that:
\begin{enumerate}
\item $T_\varnothing = T$ 
\item $\suc_{\Sigma}(\sigma) = [A_\sigma]^{<\aleph_0}$
\item $A_\sigma$ witnesses that $T_\sigma\in \Poset(\mathcal V, {\bf T})$
\item if $k_\sigma = \max\left(\bigcup \ran(\sigma)\right)+1$
and $\sigma\subseteq \tau\in \Sigma$ then
$T_\tau\restriction k_\sigma = T_\sigma\restriction k_\sigma$
\item if $m\in \bigcup \ran(\sigma)$ then
$\suc_{T_\sigma}(t) = \suc_{\bf T}(t) $
for each $t\in T_\sigma\restriction m$
\item $k_\sigma\cap A_\sigma = \varnothing$
\item  $T_{\sigma^\frown a}\angbr{t^\frown s} \in D_{k_\sigma}$ for each
$a\in \suc_{\Sigma}(\sigma)$ and $t\in T_\sigma\restriction k_\sigma$ and
\begin{align*}
 s\in \suc_{T_\sigma}(t) = \suc_{\bf T}(t).
\end{align*}
\end{enumerate}
Given $\sigma\in \Sigma$ for which $T_\sigma$ and $ A_\sigma$
have already been constructed let $a\in [A_\sigma]^{<\aleph_0}$ and note that $k_{\sigma^\frown a} = \max(a)+1$. Since $T_\sigma\in \mathfrak M$
it must also be the case that $T_\sigma\angbr{t}\in \mathfrak M$
for each $t\in T_\sigma\restriction 
k_{\sigma^\frown a}$. Therefore 
it is possible to find $\bar{T}_t\subseteq T_\sigma\angbr{t}$ such that $\bar{T}_t\in D_{k_{\sigma}}\cap \mathfrak M$.
Letting $T_{\sigma^\frown a} =\bigcup_{t\in T_\sigma\restriction 
k_{\sigma^\frown a}}\bar{T}_t$ it is immediate that $T_{\sigma^\frown a} \in \mathfrak M$.
Let $\bar{A}_t \in \mathcal V\cap \mathfrak M$ witness that $\bar{T}_t\in \Poset(\mathcal V, \bf T)$ and define 
$$A_{k_{\sigma^\frown a}} = \bigcap_{t\in T_\sigma\restriction 
k_{\sigma^\frown a}}\bar{A}_t \setminus k_{\sigma^\frown a} $$
noting that $A_{k_{\sigma^\frown a}}\in \mathfrak M$ and
that $A_{k_{\sigma^\frown a}}$  witnesses that
$T_{\sigma^\frown a} \in \Poset(\mathcal V, \bf T)$.
It is now a simple matter to check that the induction hypotheses are all satisfied and that the inductive argument can be completed. 

Since 
{\sf Player 1} has no winning strategy in the game $\Game(\mathcal V)$
it follows that there is a branch $B\subseteq \Sigma$ such that
$A_B= \bigcup_{n\in\omega}\left(\bigcup B\right)(n)\in \mathcal V$. Furthermore, $T^* = \bigcup_{\sigma\in B} T_\sigma\restriction k_\sigma$ belongs to $\Poset(\mathcal V, \bf T)$ and this is witnessed by $A_B$.
By construction, for each $\sigma^\frown a\in B$
and $t\in T^*\restriction k_\sigma$ the following hold:
\begin{itemize}
\item $\bar{T}_t = T^*_{\sigma^\frown a}\angbr{t}\in D_{k_\sigma}\cap \mathfrak M$
\item $\bar{T}_t\subseteq T\angbr{t}$
\item $T^*\subseteq \bigcup_{t\in T^*\restriction k_\sigma} \bar{T}_t$.
\end{itemize}
Since the set $\{k_\sigma\}_{\sigma\in B}$ is infinite the claim has been established.
\end{proof}
The properness  of 
$\Poset(\mathcal V, \bf T)$
 follows immediately from the claim. Given $T\in \Poset(\mathcal V, \bf T)$ and an elementary submodel $\mathfrak M$ such that $T\in \mathfrak M$ let $\{D_n\}_{n\in\omega}$ enumerate all the dense subsets of $\Poset(\mathcal V, \bf T)$ in $\mathfrak M$ and let
 $D_n^*=\bigcap_{j\in n}D_j$. Applying the claim to
 $T$ and $\{D^*_n\}_{n\in\omega}$ yields $T^*\subseteq T$ that
 is $\mathfrak M$-generic.

 To see that $\Poset(\mathcal V, \bf T)$ satisfies the Weak Sack Property suppose that
  $$T\forces{\Poset(\mathcal V, \bf T)}{\dot{f}:\omega\to\omega}$$ and that $\lim_{n\to\infty} g(n) = \infty$.
 For each $n\in \omega$ let $J(n)$ be so large that
$g(J(n))> |{\bf T}\restriction n|$.
Then let $D_n$ be the set of all $T$ such that there is some $m_T$ so that
$$T\forces{\Poset(\mathcal V, \bf T)}{\dot{f}(J(\check{n})) = \check{m}_T} .$$
Applying the claim yields a $T^*\subseteq T$ and an infinite
$A\subseteq \omega$ such that $T^*\angbr{t}\in D_n$ for each $n\in A$ and $t\in T^*\restriction n$.
Letting $G(n) = \SetOf{m_{T^*\angbr{t}}}{t\in T^*\restriction n}$ it follows that
$$T^*\forces{\Poset(\mathcal V, \bf T) }{(\forall n\in A) \ \dot{f}(J(n)) \in G(n)}$$
and, moreover, $$|G(n)|\leq |T^*\restriction n|\leq |{\bf T}\restriction n| < g(J(n))$$ for each $n\in A$. Therefore $\{J(n)\}_{n\in A}$ witnesses that the appropriate instance of the Weak Sack Property holds. 
\end{proof}

It is worth remarking that if $\bf T$ is a tree such that
$\lim_{t\in \bf T}|\suc_{\bf T}(t)|= \infty$ then $\Poset(\mathcal V, \bf T) $ does not have the Sacks Property. To see this let $g:\omega\to \omega$ be the function  defined by
$$g(n) = \min_{t\in \bf T\restriction n}|\textstyle{\suc_{\bf T}(t)}| - 1$$
and note that $\lim_{n\to\infty}g(n) = \infty$.
However, if $G\subseteq \Poset(\mathcal V, \bf T) $ is generic then 
there is no $d\in \prod_{n\in\omega}[\omega]^{g(n)}$ such that
$\left(\bigcup\bigcap G\right)(n) \in d(n)$ for all $n$ since given any
$T\in \Poset(\mathcal V, \bf T) $ there is some $t\in T$ such that
$|\suc_T(t)| = |\suc_{\bf T}(t)|>g(|t|)$. 
It is then possible to find $s\in \suc_T(t)\setminus d(|t|)$ and then
$T\angbr{t^\frown s}\forces{\Poset(\mathcal V, \bf T)}{\left(\bigcup\bigcap G\right)(|t|) = s \notin d(|t|)}$.
\subsection{Adding a $(2,\aleph_0)$-selective ultrafilter}
\label{Lasg5r}This subsection will apply the partial order of \S\ref{SSag} to add a thick set $Y\subseteq [\omega]^2$ so that for a given partition of
$[\omega]^d= A_0 \cup A_1$ either $\lambda^{d-2}(Y)\subseteq A_0$ or 
$\lambda^{d-2}(Y)\subseteq A_1$. 
\begin{notat}
Given an indexed family of sets $\{S_x\}_{x\in X}$
define $\oplus_{x\in X}S_x = \bigcup_{x\in X}\{x\}\times S_x$. In particular $\oplus_{n\in\omega} n = \bigcup_{n\in\omega}\{n\}\times n$.
\end{notat}

 \begin{defin}\label{Dfif9tehgPo667}
If $d\geq 3$ and $P: \oplus_{n\in\omega} [n]^d \to 2$ and $j\in 2$ 
then for $A\subseteq \oplus_{n\in\omega}n$ define 
$${\bf L}_{P,m,j}(A) = \SetOf{H\subseteq m}{\{m\}\times H\subseteq A\And (\forall s\in [H]^d)) \  P(m,s) = j } $$
 let
 $\mu_{P,m,j}(A) = \max_{H\in {\bf L}_{P,m,j}(A)}(|H|)$.
 Denote  ${\bf L}_{P,m,j}(\oplus_{n\in\omega}n)$ by $ {\bf L}_{P,m,j}$.
 \end{defin}

\begin{defin}\label{Thjyg6}
If $\mathcal V$ is a filter 
on $\omega$
define  a filter $\mathcal U$ on $\oplus_{n\in\omega} n$
 to be $\mathcal V$-thick if
$$\SetOf{n\in\omega}{k\leq |\SetOf{j\in n}{(n,j)\in A}|}\in \mathcal V^+$$ for every $A\in \mathcal U$ and 
$k\in\omega$.
\end{defin}
\begin{notat}
Let $R^m(k)$ be the least  integer such that
$R^m(k)\rightarrow (k)^m_2$.
\end{notat}
\begin{lemma}\label{Desjhtyf6}
If
 $\mathcal U$ is a filter on $\oplus_{n\in\omega} n$
that is $\mathcal V$-thick
and $P: \oplus_{n\in\omega} [n]^d \to 2$
then there is $J\in 2$ such that
$$(\forall A\in \mathcal U)(\forall k\in\omega) \ \SetOf{m\in \omega}{\mu_{P,m,J}(A)\geq k}\in \mathcal V^+$$
\end{lemma}
\begin{proof}
	If the lemma fails then there are $A\in \mathcal U$ and $k\in\omega$ such that $$Z= \SetOf{m\in\omega}{\mu_{P,m,0}(A)< k}\cap \SetOf{m\in\omega}{\mu_{P,m,1}(A)< k} \in \mathcal V .$$  However, since
	$\mathcal U$ is  $\mathcal V$-thick
	there must be some $m\in Z$ such that
	$$R^d(k)< |\SetOf{\ell\in m}{(m,\ell)\in A}| . $$
	But then by Ramsey's Theorem there is some $j\in 2$ and $H\subseteq \SetOf{\ell\in m}{(m,\ell)\in A}$ such that $|H|\geq k$ and $P(m,s) = j$ for all $s\in [H]^d$.
	Hence $\mu_{P,m,j}(A)\geq |H|\geq k$
	contradicting that $m\in Z$.
	\end{proof}
\begin{defin}\label{DefDeds491687}
If
$\mathcal U$ is a filter on $\oplus_{n\in\omega} n$
that is $\mathcal V$-thick
and $P: \oplus_{n\in\omega} [n]^d \to 2$
then use Lemma~\ref{Desjhtyf6} to let $J\in 2$ be the least such that
$$(\forall A\in \mathcal U)(\forall k\in\omega) \ \SetOf{m\in \omega}{\mu_{P,m,J}(A)\geq k}\in \mathcal V^+ $$	and define ${\bf T}_{\mathcal U, \mathcal V, P}$ to be the tree
$$\bigcup_{n\in\omega}\prod_{m\in n}{\bf L}_{P,m,J} .$$

	\end{defin}	
	
	\begin{corol}\label{corjwithdd5}
	If 
	$\mathcal U$ is a filter on $\oplus_{n\in\omega} n$
	that is $\mathcal V$-thick
	and $P: \oplus_{n\in\omega} [n]^d \to 2$ then
	$G\subseteq \Poset(\mathcal V, {\bf T}_{\mathcal U, \mathcal V, P})$ is generic
	then  $\{\bigcup\bigcap G \}\cup \mathcal U$ generates a $\mathcal V$-thick filter and 
	$P(m,s) = J$ for each	$m\in\omega$ and $s\in \left[\bigcap G(m)\right]^3$.
	
		\end{corol}
\begin{proof}
	Let $U\in \mathcal U$, $k\in \omega$ and $V\in \mathcal V$. It must be shown that 
		\begin{equation}\label{kjas6}
	 \SetOf{n\in\omega}{k\leq \left|\SetOf{j\in n}{(n,j)\in U\cap \bigcup\bigcap G}\right|}\cap V\neq \varnothing .
	\end{equation}
	To this end suppose that  $T\in \Poset(\mathcal V, {\bf T}_{\mathcal U, \mathcal V, P})$ and this is witnessed by $B\in \mathcal V$. 
	It follows from Definition~\ref{DefDeds491687} that 
	$$ \SetOf{m\in \omega}{\mu_{P,m,J}(U)\geq k}\cap V\cap B\neq \varnothing $$
	 and so there is 
	$m\in B\cap V$ and $t\in T\restriction m$ such that
	$\suc_T(t) = {\bf L}_{P,m,J}$.  Since $\mu_{P,m,J}(U)\geq k$ there is some $H\subseteq m$ such that
	\begin{itemize}
		\item $|H| = k$ 
		\item  $P(m,s) = J$ for all $s\in [H]^d$ 
		\item 
	$\{m\}\times H\subseteq U$.
	\end{itemize}
	Then $H\in {\bf L}_{P,m,J}$ and so $$T\angbr{t^\frown H}\forces{\Poset(\mathcal V, {\bf T}_{\mathcal U, \mathcal V, P})}{H\subseteq \left\{\bigcup\bigcap G \right\} }$$
	and since $\{m\}\times H\subseteq U$ it follows that $m$ witnesses that Inequality~(\ref{kjas6}) holds.
	\end{proof}

Corollary~\ref{corjwithdd5} can now be used in conjunction
with the proof of Corollary~\ref{CorNoNowjlhMyV} to obtain a model with no nowhere dense ultrafilters, but still containing a $(2,\aleph_0)$-selective ultrafilter.
\begin{corol}
	\label{laj7}
	It is consistent that there is a  $(2,\aleph_0)$-selective ultrafilter yet there are no nowhere dense ultrafilters.
\end{corol}
\begin{proof}
	As in the proof of Corollary~\ref{CorNoNowjlhMyV}
	let $V$ be a model of $2^{\aleph_0} = \aleph_1$ such that $\lozenge_{\omega_2,\omega_1}$ holds and this is witnessed by $\{D_\alpha\}_{\alpha\in\omega_2} $.
	Suppose also that  $\mathcal V$ is a P-point in the ground model $V$. Let
	$W= \SetOf{\alpha\in\omega_2}{\cof{\alpha}\neq \omega_1}$ and let $\{\dot{P}_\alpha\}_{\alpha\in W}$ enumerate cofinally often the set
	$$\SetOf{\dot{P}}{\mathbb Q\in [\mathsf{H}(\aleph_2)]^{\aleph_1}\And {\sf 1}\forces{\mathbb Q}
		{\dot{P}: \oplus_{n\in\omega}[n]^d \to 2}} .$$
	Construct a countable support iteration	 $\{\Poset_\alpha\}_{\alpha\in\omega_2}$ and $\dot{\mathcal U}_\alpha$  such that for each $\alpha\in \omega_2$:
	\begin{enumerate}
		\item \label{Onh111}$\mathcal U_0$ is the filter on $[\oplus_{n\in\omega}n]^2 $ generated by $\oplus_{n\in\omega}[n]^2$ and the co-finite subsets of $[\oplus_{n\in\omega}n]^2$
	\item \label{CerHerd4}	$\mathbb P_\alpha\in [{\mathsf H}(\aleph_2)]^{\aleph_1}$
	\item\label{Cert3} ${\sf 1}\forces{\Poset_\alpha}{\dot{\mathcal U}_\alpha \text{ is a $\mathcal V$-thick filter on } [\oplus_{n\in\omega}n]^2}$
	\item\label{Cert34} ${\sf 1}\forces{\Poset_\alpha}{(\forall \beta \in \alpha) \ \dot{\mathcal U}_\alpha \supseteq \dot{\mathcal U}_\beta}$
	\item  if $\cof(\alpha) = \omega_1$ and  $D_\alpha$
	is a $\Poset_\alpha$ name such that ${\sf 1}\forces{\Poset_\alpha}{D_\alpha\text{ is a maximal ideal}}$ then
	$\Poset_{\alpha+1} = \Poset_\alpha*\Poset(D_\alpha)$ where $\Poset(D_\alpha)$ is as in Definition~\ref{Sheas31ah7ib}
	\item   if $\cof(\alpha)\neq \omega_1$ and 
	${\sf 1}\forces{\Poset_\alpha}{\dot{P}_\alpha: \oplus_{n\in\omega}[n]^{d_\alpha} \to 2}$
then
	 $\Poset_{\alpha+1} = \Poset_\alpha* \Poset(\mathcal V, {\bf T}_{\mathcal U_\alpha, \mathcal V, P_\alpha})$
	 \item if none of the preceding hold then
	 $\Poset_{\alpha+1} = \Poset_\alpha$.
	
	\end{enumerate} 
	It will first be shown by recursion on $\alpha$ that each $\Poset_\alpha$ is proper and enjoys the Weak Sacks Property.
	
	If $\cof(\alpha) = \omega_1$ and $\Poset_\alpha$ is proper and satisfies the Weak Sacks Property
	and $${\sf 1}\forces{\Poset_\alpha}{D_\alpha \text{ is a maximal ideal on }\omega}$$
	 then
	it follows from Theorem~\ref{ShePrewwsA}, Corollary~\ref{CCOOR777hDrh75} and Corollary~\ref{CCOOR777hDrh75A} 
	that $\Poset_{\alpha+1} =\Poset_\alpha\ast \Poset(D_\alpha)$ is proper and has the Weak Sacks Property and, of course, (\ref{CerHerd4}) holds.  In this case let $\dot{\mathcal U}_{\alpha+1} = \dot{\mathcal U}_\alpha$ and note that $\mathcal U_{\alpha+1}$  satisfies (\ref{Cert3}) and (\ref{Cert34})
	holds.
	
	If $\cof(\alpha) \neq \omega_1$ and $\Poset_\alpha$ is proper and satisfies the Weak Sacks Property
	and 
		$${\sf 1}\forces{\Poset_\alpha}{\dot{P}_\alpha: \oplus_{n\in\omega}[n]^{d_\alpha} \to 2}$$
	then
	an important point to note is that
	Corollary~\ref{Fhvcttar5} guarantees that $\mathcal V$
	generates a non-meagre filter after forcing with $\Poset_\alpha$. 
 And the fact that ${\mathbb{P}}_{\alpha}$ is proper guarantees that the filter generated by $\mathcal{V}$ after forcing with ${\mathbb{P}}_{\alpha}$ is a P-filter.
 It then follows from Lemma~\ref{Lakk6hy7} that
	$\Poset_{\alpha+1} =\Poset_\alpha\ast \Poset(\mathcal V, {\mathbf T}_{\mathcal U_\alpha, \mathcal V, P_\alpha})$ is proper and has the Sacks Property and, as before, (\ref{CerHerd4}) holds.
	 In this case let $\dot{G}_\alpha$ be a name for the generic subset of
	 $\Poset(\mathcal V, {\mathbf T}_{\mathcal U_\alpha, \mathcal V, P_\alpha})
	 $
and let	   $\dot{\mathcal U}_{\alpha+1} $ be a name for the filter generated by $\dot{\mathcal U}_\alpha\cup \{\bigcup\bigcap \dot{G}_\alpha\}$ and 
use Corollary~\ref{corjwithdd5} to conclude that
(\ref{Cert3}) holds. It is immediate that (\ref{Cert34}) also holds.

If $\alpha$ is a limit and $\Poset_\beta$ and $\mathcal U_\beta $ have been defined for all $\beta \in \alpha$ and satisfy the induction hypothesis, let
$\Poset_\alpha$ be the countable support limit of the $\Poset_\beta$.
It follows from Theorem~\ref{ShePrewws} that $\Poset_\alpha$ is proper and has the Weak Sacks Property.
	Since $\mathrm{CH}$ holds in $V$ it also follows that (\ref{CerHerd4}) holds.  
	Since $\Poset_\beta$ is completely embedded in $\Poset_\alpha$ it is possible to consider each $\mathcal U_\beta$ as a $\Poset_\alpha$-name
	 and to let $\dot{\mathcal U}_\alpha$ be the $\Poset_\alpha$-name  $ \bigcup_{\beta\in \alpha}\dot{\mathcal U}_\beta$.
	It is routine to see that $\dot{\mathcal U}_\alpha $
	satisfies (\ref{Cert3}) and (\ref{Cert34}).

The proof that forcing with $\Poset_{\omega_2}$ produces a model with no nowhere dense ultrafilters is now exactly the same as the proof of Corollary~\ref{CorNoNowjlhMyV}. 
	If $G\subseteq \Poset_{\omega_2}$ is generic let  $\mathcal U = \bigcup_{\alpha\in\omega_2}\dot{\mathcal U}_\alpha[G] $. To see that $\mathcal U$ is $(2,\aleph_0)$-selective in the model $V[G]$ let
	$\bar{P}:[\oplus_{n\in\omega}n]^d\to 2$.
	 Because of (\ref{Onh111}) the filters
	$\dot{\mathcal U}_\alpha[G]$ can be thought of as filters on $\oplus_{n\in\omega}[n]^2$ and so let
	$P = \bar{P}\restriction \oplus_{n\in\omega}[n]^d$. Then since $\Poset_{\omega_2}$ has the $\aleph_2$ chain condition there is some $\alpha\in \omega_2$ such that $P$ has a $\Poset_\alpha$-name $\dot{P}$. Let $\beta\in W\setminus \alpha$ be such that $\dot{P}_\beta = \dot{P}$. Let  ${G}_\beta$ be the generic subset of
	$\Poset(\mathcal V, {\mathbf T}_{\mathcal U_\beta, \mathcal V, P_\beta})
	$ obtained from $G$. Then by construction $\bigcup\bigcap {G}_\beta\in \mathcal U$.
	Moreover, $\dot{P}_\beta[G_\beta] = \dot{P}[G_\beta] = P$ and by Corollary~\ref{corjwithdd5}
	it follows that  and
	$P(m,s) = J$ for each	$m\in\omega$ and $s\in \left[\bigcap G_\beta(m)\right]^d$. In other words,
	$\lambda^{d-2}(\bigcup\bigcap G_\beta)\in \lambda^{d-2}(\mathcal U)$ and $\lambda^{d-2}(\bigcup\bigcap G_\beta)$ is homogeneous for $P$ and, hence, also for $\bar{P}$. Since $\bar{P}$
	was arbitrary it follows that $\lambda^{d-2}(\mathcal U)$ is an ultrafilter, as required.
\end{proof}
\subsection{Final remarks}
An obvious corollary to Corollary~\ref{laj7} is that the existence of a $(2,3)$-selective ultrafilter
is quite a weak hypothesis since even the stronger hypothesis of the existence of a $(2,\aleph_0)$-selective ultrafilter   does not  imply the existence of a nowhere dense ultrafilter.  This should be considered in the context of the results of \cite{MR3946661} which included examples of $(2,3)$-selective ultrafilters that are not P-points. However, since these examples were constructed using $\mathrm{CH}$, they did not rule out the existence of nowhere dense ultrafilters, indeed, not even the existence of P-points. It may therefore be worthwhile repeating the following question originally raised in \cite{MR3946661}.
\begin{quest}\label{Q1}
	Is it consistent that there are no $(2,3)$-selective ultrafilters?
	\end{quest}
	
	Probably not much easier is the following question.	
	\begin{quest}\label{Q2}
		Is it consistent that there are no $(2,\aleph_0)$-selective ultrafilters?
	\end{quest}
	
	Of technical interest is the observation the $(2,\aleph_0)$-selective ultrafilter constructed in Corollary~\ref{laj7} has the additional property that is contains a set of the form $\SetOf{[x]^2}{x\in \mathcal X}$ where $\mathcal X$ is an infinite family of pairwise disjoint finite sets. This raises the following question, which may be easier than either
	Question~\ref{Q1} or Question~\ref{Q2}.
	
	\begin{quest}
		Let $\mathcal J$ be the ideal on $[\omega]^2$ generated by sets of the form $\SetOf{[x]^2}{x\in \mathcal X}$ where $\mathcal X$ is a family of pairwise disjoint finite sets. Is it consistent that $\mathcal J\subseteq \mathcal U^*$ for every $(2,\aleph_0)$-selective ultrafilter $\mathcal U$ on $[\omega]^2$? In other words, is it consistent that there are no $(2,\aleph_0)$-selective ultrafilters of the type constructed in Corollary~\ref{laj7}?
		\end{quest}
	
	It needs to be emphasized that the results presented here have not ruled out the possibility that Shelah's model of \cite{MR1690694}, or one of its variations, might answer Question~\ref{Q1} in the positive. What has been shown is only that any argument establishing this will be quite delicate because it will have to eliminate the possibility of interleaving partial orders with very similar properties into the iteration as is done in \S\ref{Lasg5r}.

\providecommand{\bysame}{\leavevmode\hbox to3em{\hrulefill}\thinspace}
\providecommand{\MR}{\relax\ifhmode\unskip\space\fi MR }
\providecommand{\MRhref}[2]{%
  \href{http://www.ams.org/mathscinet-getitem?mr=#1}{#2}
}
\providecommand{\href}[2]{#2}

\end{document}